\documentclass[reqno,10pt,a4paper]{amsart}
\usepackage{amssymb}
\usepackage{amsthm}
\usepackage{enumitem}
\usepackage{lmodern}
\usepackage[usenames,dvipsnames]{color}
\usepackage{hyperref}
\hypersetup{
 colorlinks,
 linkcolor={red!50!black},
 citecolor={blue!50!black},
 urlcolor={blue!80!black}
}
\usepackage[a4paper, inner=3cm,outer=3cm,top=3.3cm,bottom=3cm]{geometry}
\usepackage{subcaption}
\usepackage{graphicx}
\usepackage{pgfplots}
\usepackage{tikz,tikz-3dplot}
\usetikzlibrary{cd}
\usetikzlibrary{shapes.misc}
\usetikzlibrary{arrows,decorations.markings}
\usepackage{xifthen}
\usepackage{todonotes}

\AtBeginDocument{%
\def\MR#1{}
}
\theoremstyle{plain}
\newtheorem{thm}{Theorem}[section]
\newtheorem{prop}[thm]{Proposition}
\newtheorem{cor}[thm]{Corollary}
\newtheorem{lem}[thm]{Lemma}

\newtheorem{quest}[thm]{Question}
\theoremstyle{definition}
\newtheorem{defn}[thm]{Definition}
\newtheorem{ex}[thm]{Example}
\newtheorem{rem}[thm]{Remark}

\newcommand{\vol}{{\mathrm{vol}}}

\newcommand{\sqp}[1][]{%
\ifthenelse{\isempty{#1}}{\mathcal{P}(\square_2)}{\mathcal{P}^{#1}(\square_2)}%
}

\DeclareMathOperator{\conv}{conv}

\DeclareMathOperator{\md}{md}

\DeclareMathOperator{\MV}{MV}

\newcommand{\R}{{\mathbb{R}}}
\newcommand{\Z}{{\mathbb{Z}}}

\newcommand{\Pm}{{\mathcal{P}}}

\newcommand{\D}{\Delta}

\newcommand{\rleft}{\mathopen{}\mathclose\bgroup\left}
\newcommand{\rright}{\aftergroup\egroup\right}

\newcommand{\set}[1]{\rleft\{ {#1} \rright\}}

\newcommand{\bzero}{\textnormal{\textbf{0}}}
\newcommand{\be}{\textnormal{\textbf{e}}}

\def\coloneqq{\mathrel{\mathop:}=}
\title{Families of lattice polytopes of mixed degree one}

\author[G.\,Balletti]{Gabriele Balletti}
\address[G.\,Balletti]{Department of Mathematics\\Stockholm University\\SE-$106$\ $91$\
  Stockholm\\Sweden}
\curraddr{}
\email{gabriele.balletti@gmail.com}
\thanks{}

\author[C.\,Borger]{Christopher Borger}
\address[C.\,Borger]{Fakult\"at f\"ur Mathematik\\Institut f\"ur Algebra und Geometrie\\Otto-von-Guericke-Universit\"at Magdeburg\\Universit\"atsplatz 2\\ 39106 Magdeburg\\Germany}
\curraddr{}
\email{christopher.borger@ovgu.de}
\thanks{}


\begin{document}

\begin{abstract}
It has been shown by Soprunov that the normalized mixed volume 
(minus one) of an $n$-tuple of $n$-dimensional lattice polytopes
is a lower bound for the number of interior lattice points in 
the Minkowski sum of the polytopes. He defined $n$-tuples of
mixed degree at most one to be exactly those for which this
lower bound is attained with equality,
and posed the problem of a classification of such tuples. 
We give a finiteness result regarding this problem in general dimension $n \geq 4$, showing that all but finitely many
$n$-tuples of mixed degree at most one admit a common
lattice projection onto the unimodular simplex $\D_{n-1}$.
Furthermore, we give a complete solution in dimension $n=3$.
In the course of this we show that our finiteness result does
not extend to dimension $n=3$, as we describe infinite families of triples of mixed degree one not admitting a common
lattice projection onto the unimodular triangle $\D_2$.
\end{abstract}

\maketitle{}

\section{Introduction}
\label{sec:introduction}


\subsection{Basic definitions}
A \emph{lattice polytope} $P \subset \R^n$ is a polytope $P \subset \R^n$
whose vertices are elements of the lattice $\Z^n \subset \R^n$.
We call two lattice polytopes $P_1,P_2 \subset \R^n$ \emph{equivalent} if 
there exists an affine lattice-preserving transformation $U \colon
\R^n \to \R^n$ satisfying $U(P_1)=P_2$. 
We say that two $n$-tuples $P_1,\dots,P_n \subset \R^n$ and $Q_1,\dots,Q_n \subset \R^n$ 
are \emph{equivalent} if there is a permutation $\sigma \in S_n$, an affine lattice-preserving transformation $U \colon \R^n \to \R^n$ and vectors $t_1,\dots,t_n \in \Z^n$
such that $(P_1,\dots,P_n) = (U(Q_{\sigma(1)})+t_1),\dots,U(Q_{\sigma(n)})+t_n))$. We denote by $\Delta_n = \conv(\bzero,\be_1,\dots,\be_n) \subset \R^n$
the \emph{standard unimodular simplex} in $\R^n$ and call an 
$n$-dimensional simplex \emph{unimodular} if it is equivalent to
$\Delta_n$. 
We write the \emph{Minkowski sum} of two lattice polytopes 
$P_1,P_2 \subset \R^n$ as $P_1 + P_2 = \set{p_1+p_2 \colon p_1 \in P_1, 
p_2 \in P_2} \subset \R^n$ and denote the \emph{interior} of a lattice polytope
$P \subset \R^n$ by $P^{\circ}$. If one has $P^{\circ} \cap \Z^n = \emptyset$, we call the lattice polytope $P \subset \R^n$ \emph{hollow}. 

\subsection{Motivation}
In order to give an explicit definition, let us define the
(normalized) \emph{mixed volume} of an $n$-tuple of polytopes $P_1,\dots,P_n \subset \R^n$ via the inclusion-exclusion formula given
by
$\MV(P_1,\dots,P_n) \coloneqq \sum_{\emptyset \neq I \subseteq \set{1,\dots,n}} (-1)^{n-|I|} \vol_n(\sum_{i \in I} P_i)$, where $\vol_n$ denotes the standard euclidean volume in $\R^n$.
Note that there are various equivalent definitions for the mixed volume of 
an $n$-tuple of lattice polytopes or, more generally, for an $n$-tuple of convex bodies in $\R^n$ (see e.g. \cite{Schn14} or \cite{EG15}) and
that in our definition the mixed volume is normalized such that
$\MV(\Delta_n,\dots,\Delta_n)=1$. 
A central connection of the mixed volume to 
algebraic geometry is given by the famous Bernstein-Kouchnirenko-Khovanskii theorem (\cite{B75}). Combining this theorem with a generalization of the
Euler-Jacobi theorem due to Khovanskii (\cite{K78}) in the context
of sparse polynomial interpolation, Soprunov showed
the following lower bound on the number of interior lattice points
in the Minkowski sum of an $n$-tuple of $n$-dimensional lattice polytopes.

\begin{thm}[\cite{S07,Nil17}]
\label{thm:bound}
Let $P_1,\dots,P_n \subset \R^n$ be $n$-dimensional lattice polytopes. 
Then the following inequality holds:
\begin{align*}
|(P_1+\dots+P_n)^{\circ} \cap \Z^n| \geq \MV(P_1,\dots,P_n)-1.
\end{align*}
Furthermore, equality holds if and only if the Minkowski sum of any
choice of $n-1$ polytopes of the tuple $P_1,\dots,P_n$ is hollow. 
\end{thm}

The $n$-tuples for which equality holds in the above theorem
have been called $n$-tuples \emph{of mixed degree at most one} by Soprunov in 
\cite{BNR08}, where a characterization of such tuples has 
been posed as a problem (\cite[Section~5, Problem~2]{BNR08}). This 
notion is motivated by a connection to the \emph{degree}
of a lattice polytope, which is
an intensively studied invariant in Ehrhart Theory (see for example
\cite{BN07,DdRP09,dRHN11,B18}). The degree $\deg(P)$
of an $n$-dimensional lattice polytope $P \subset \R^n$ is set to 
equal $n$
if $P$ has at least one interior lattice point, and otherwise is defined as the
smallest integer $0 \leq d \leq n-1$ such that the
dilated lattice polytope $(n-d)P$ is hollow. Another 
interpretation is given by the fact that 
$\deg(P)$ agrees with the degree of the so-called $h^*$-polynomial of $P$
(see for example \cite{BR15}). 
Now in the setting
$P_1=\dots=P_n=P \subset \R^n$, that is for an $n$-tuple consisting of
$n$ copies of the same lattice polytope $P$, the equality condition 
from Theorem~\ref{thm:bound} is satisfied if and only if the degree
of $P$ is at most one. It is a well-known fact that an
$n$-dimensional lattice 
polytope has degree $0$ if and only if it is equivalent to $\Delta_n$.
Also for lattice polytopes of degree one there exists the following
complete description by Batyrev-Nill~\cite{BN07}. Given an $n$-dimensional lattice polytope $Q \subset \R^n$, we define the \emph{lattice pyramid} $\Pm(Q)$ as the
$(n+1)$-dimensional polytope
\[
\Pm(Q) \coloneqq \conv(Q \times \{\bzero\} \cup \{\be_{n+1}\}) \subset \R^{n+1}. 
\]
The lattice pyramid construction preserves the degree (it actually preserves the $h^*$-polynomial). 
We say that an $n$-dimensional lattice polytope is an \emph{exceptional
    simplex} if it is equivalent to the polytope obtained via $n-2$ iterations of the lattice pyramid construction over the 
polygon $2 \Delta_2$.
We say that an $n$-dimensional lattice polytope $P$ is a \emph{Lawrence prism} if $P$ is equivalent to a lattice polytope $\conv (\{ \bzero , a_0 \be_n , \be_1 , \be_1 + a_1 \be_n , \ldots ,
\be_{n-1}, \be_{n-1} + a_{n-1} \be_n  \})$ for nonnegative integers 
$a_0, \ldots , a_{n-1} \in \Z_{\geq 0}$.

\begin{thm}[{\cite[Theorem~2.5]{BN07}}]
    \label{thm:BN}
    Let $P$ be a lattice polytope. Then $\deg(P) \leq 1$
    (i.e. $(n-1)P$ is hollow) if and only if $P$ is
    is an exceptional simplex or a Lawrence prism.
\end{thm}

The relation of tuples of mixed degree at most one to the degree
of a lattice polytope raises the natural question whether there is a general
concept of a mixed degree of an $n$-tuple of lattice polytopes in $\R^n$ 
that generalizes both Soprunov's definition of tuples of mixed degree
at most one and the degree of a single lattice polytope. 
A suggestion for such a mixed degree has recently been
given by Nill \cite{Nil17}.
Let $P_1,\dots,P_n \subset \R^n$ be an $n$-tuple of lattice polytopes.
The \emph{mixed degree} $\md(P_1,\dots,P_n)$ is set to equal $n$ if 
$P_i$ has an interior lattice point for some $1 \leq i \leq n$.
Otherwise $\md(P_1,\dots,P_n)$ is the 
smallest integer $0 \leq d \leq n-1$ such that  
the Minkowski sum of any choice of $(n-d)$ polytopes of tuple
$P_1,\dots,P_n$ is hollow. We refer the reader to \cite{Nil17} for 
additional motivation for this definition.

In this language, Soprunov's problem asks for a 
characterization of $n$-tuples of lattice polytopes 
$P_1,\dots,P_n \subset \R^n$ satisfying $\md(P_1,\dots,P_n) \leq 1$. 
The case of $\md(P_1,\dots,P_n)=0$ (as this is equivalent to 
$\MV(P_1,\dots,P_n)=1$ by \cite[Theorem 2.2]{Nil17}) has already been solved by 
Cattani et al. in the context of investigating the codimension
of so-called mixed discriminants.

\begin{prop}[{\cite[Proposition~2.7]{CCD13}}]
	\label{prop:ccd}
    Let $P_1,\ldots,P_n$ be $n$-dimensional lattice polytopes.
    Then $\md(P_1,\ldots,P_n)=0$ if and only if the $n$-tuple
    $P_1,\ldots,P_n$ is equivalent to the $n$-tuple
    $\Delta_n,\dots,\Delta_n$.
\end{prop}

We therefore often restrict to tuples with mixed degree equal to one
in our approach towards solving Soprunov's problem.

\subsection{Results}

The contribution of this paper is to partially solve Soprunov's problem by presenting a finiteness result for dimension $n \geq 4$ and to give a complete characterization of 
triples of $3$-dimensional lattice polytopes of mixed degree one. 
 
In order to describe a trivial class of $n$-tuples of mixed degree 
(at most) one, let us introduce the concept of lattice projections.
By \emph{lattice projection}, we denote a surjective affine-linear
map $\varphi \colon \R^n \to \R^m$ satisfying $\varphi(\Z^n)=\Z^m$.
The kernel of such a projection is affinely generated by lattice points
of $\Z^n$ and we consider two projections to be equal if and only if they
have the same kernel up to lattice translations.

The trivial class of $n$-tuples of mixed degree (at most) one 
is now given by the following example.

\begin{ex}
    \label{ex:lawrence}
    Let $P_1,\dots,P_n \subset \R^n$ be $n$-dimensional lattice
    polytopes and $\varphi \colon \R^n \to \R^{n-1}$ a
    lattice projection satisfying
    $\varphi(P_i)=\Delta_{n-1}$ for all $1 \leq i \leq n$. Then $P_1,\ldots,P_n$ has mixed degree at most
    one, as any Minkowski sum of $n-1$ polytopes of the
    tuple $P_1,\ldots,P_n$
    is projected onto the hollow polytope $(n-1)\Delta_{n-1}
    \subset \R^{n-1}$ by $\varphi$.
    An example of such a trivial tuple in dimension $n=3$ is shown in Figure~\ref{fig:example_trivial}.
\end{ex}

\begin{figure}[!ht]
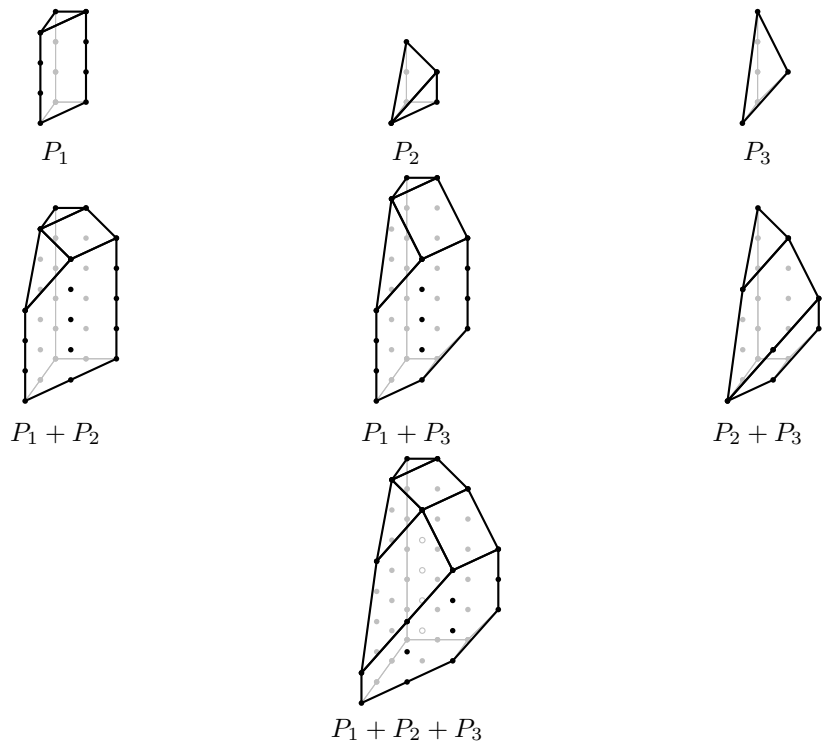

    \centering
    \begin{minipage}[b]{0.3\linewidth}
        \centering

    \end{minipage}
    \caption{A triple $P_1,P_2,P_3 \subset \R^3$ 
    having mixed degree one, where $P_1,P_2,P_3$ all project
    onto $\D_2$ under the projection along the vertical
    axis.}
    \label{fig:example_trivial}
\end{figure}

One can view $n$-tuples from Example~\ref{ex:lawrence} as 
consisting of $n$ lattice polytopes that all are Lawrence prisms
and additionally satisfy that they extend into the same 
height-direction over the same unimodular $(n-1)$-dimensional
simplex. 
Clearly we cannot expect this to be the only
class of $n$-tuples of mixed degree one, as already the 
unmixed setting of Theorem~\ref{thm:BN} additionally yields 
$n$-tuples of copies of the same exceptional simplex as
having mixed degree one. 
Unlike in the unmixed case there actually exist many more
such non-trivial examples
(see our classification result for $n=3$ in Theorem~\ref{thm:classification3d}, one example is shown in Figure~\ref{fig:example}).

\begin{figure}[!ht]
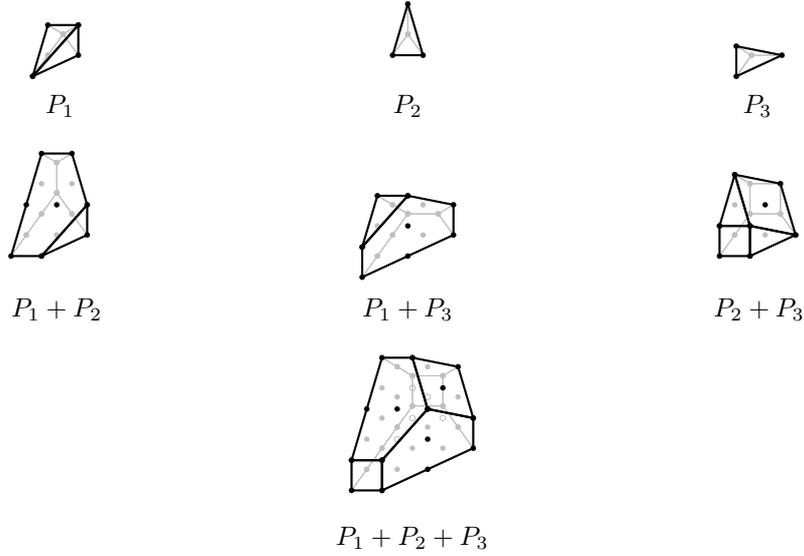

    \centering
    \begin{minipage}[b]{0.3\linewidth}
        \centering

    \end{minipage}
    \caption{A triple $P_1,P_2,P_3 \subset \R^3$ 
    having mixed degree one for which no lattice 
    projection exists commonly mapping $P_1,P_2,P_3$ onto 
    translates of $\D_2$ (see \ref{itm:max_triple_picture} of
    Corollary~\ref{thm:maximal}).}
    \label{fig:example}
\end{figure}

This raises the question whether there is any chance to 
make reasonable statements about $n$-tuples of mixed
degree one at all. Our main result is to provide a positive
answer to this question by showing that, for any dimension $n$ 
at least $4$, all but finitely many exceptions of $n$-tuples
of mixed degree one are actually of the trivial type
described in Example~\ref{ex:lawrence}.


\begin{thm}
\label{thm:main}
Fix $n \geq 4$ and let $P_1,\dots,P_n \subset \R^n$ be $n$-dimensional lattice polytopes
with $\md(P_1,\dots,P_n)=1$. Then, up to equivalence, the $n$-tuple $P_1,\dots,P_n$ either belongs to a finite list 
of exceptions or there is a lattice projection $\varphi \colon \R^n \to \R^{n-1}$ such that
$\varphi(P_i) = \D_{n-1}$ for all $1 \leq i \leq n$.
\end{thm}

We refer the reader to Section~3 for the proof of Theorem~\ref{thm:main}.

Theorem~\ref{thm:main} is not true for dimension $n \in \set{2,3}$. This
fact is straightforward to see for $n=2$, as pairs of lattice polygons
$P_1,P_2 \subset \R^2$ are of mixed degree (at most) one if and only if
both $P_1$ and $P_2$ are hollow. Fixing $P_1$ to be any hollow polygon and 
letting $P_2$ range through all polygons that are equivalent to a fixed
hollow polygon will clearly yield infinitely many non-equivalent pairs
of mixed degree one without there being a projection
commonly mapping both
polytopes onto the segment $\Delta_1$. 

For $n=3$, however, we find that only a very specific class of triples
of mixed degree one contains an infinite number of exceptions and we
can explicitly describe a finite number of $1$-parameter families 
covering this class. This is part of the following classification
result, which essentially gives a complete answer
to Soprunov's problem for dimension $n=3$. 
Note that we say that a 
$k$-tuple of $n$-dimensional lattice polytopes $P_1,\dots,P_k$ 
\emph{admits a common lattice projection} onto translates of an
$(n-1)$-dimensional lattice polytope $Q$ if there exists a lattice 
projection $\varphi \colon \R^n \to \R^{n-1}$ satisfying 
$\varphi(P_i)=Q+t_i$ for all $1 \leq i \leq k$ and some $t_i \in \Z^{n-1}$.

\begin{thm}
	\label{thm:classification3d}
    Let $P_1,P_2,P_3 \subset \R^3$ be a triple of $3$-dimensional lattice polytopes that satisfies $\md(P_1,P_2,P_3)=1$. Then either there is a lattice projection
    $\varphi \colon \R^3 \to \R^{2}$ such that $\varphi(P_i) = \D_{2}$ for all $1 \leq i \leq 3$,
    or one of the following holds.
    \begin{enumerate}
        \item There is no pair in $P_1,P_2,P_3$ admitting a common lattice projection
        onto translates of $\Delta_2$ and $P_1,P_2,P_3$
        is equivalent to one out of $29$ possible triples.
        \item There is exactly one pair in $P_1,P_2,P_3$ 
        admitting a common lattice
        projection onto translates of $\Delta_2$ and
        $P_1,P_2,P_3$ is equivalent to one out of $141$ possible triples.
        \item There are exactly two pairs in $P_1,P_2,P_3$ 
        admitting
        a common lattice
        projection onto translates of $\Delta_2$ and
        $P_1,P_2,P_3$ is equivalent to one out of $82$ possible triples.
        \item
		\label{item:infinite_families}        
        All pairs in $P_1,P_2,P_3$ admit a common lattice
        projection onto translates of $\Delta_2$ and
        \begin{enumerate}
       	\item the kernels of the projections cannot be shifted into a common
       	hyperplane and the triple
       	 $P_1,P_2,P_3$ is equivalent to one out of 
       	$27$ possible triples.
       	\item the kernels of the projections can be shifted into a common
       	hyperplane and $P_1,P_2,P_3$ belongs, up to equivalence, to one out of
       	finitely many infinite 1-parameter families of triples.
       	\end{enumerate}
    \end{enumerate}
\end{thm}

We refer the reader to Section~4 for a proof of Theorem~\ref{thm:classification3d}.

In the following example we present one of the
$1$-parameter families from Theorem~\ref{thm:classification3d}~\eqref{item:infinite_families}.
Let us denote $\square_2 \coloneqq \conv(\bzero,\be_1,\be_2,\be_1+\be_2) 
\subset \R^2$.

\begin{ex}
\label{ex:special}
Let $P^k_1,P^k_2,P_3 \subset \R^3$ be the triple given by
		\begin{itemize}[label={}]
			\item $P^k_1 \coloneqq \conv(\be_1, \be_2, \be_1 + \be_2, 2\be_2, k  \be_2 + \be_3)$,
			\item $P^k_2 \coloneqq \conv(\be_1, \be_2, \be_1 + \be_2, 2\be_1, k  \be_1 + \be_3)$,
			\item $P_3 \coloneqq \conv(\bzero,\be_1, \be_2, \be_1 + \be_2, \be_3)$,
		\end{itemize}
for some $k \in \Z_{\geq 0}$. Then $\md(P^k_1,P^k_2,P_3)=1$ and,
while all pairs in $P^k_1,P^k_2,P_3$ admit a common lattice projection
onto translates of $\D_2$, there is no lattice projection
commonly mapping the whole triple
$P^k_1,P^k_2,P_3$ onto translates of $\D_2$.
Note that $P^k_1, P^k_2$ and $P_3$ as single lattice polytopes are all equivalent to $\Pm(\square_2)$ for all 
$k \in \Z_{\geq 0}$ (see Figure~\ref{fig:ex_family}).
\end{ex}

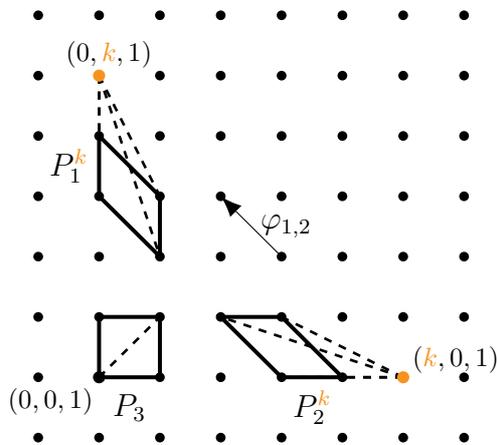
\begin{figure}[h]
\begin{tikzpicture}[scale=0.8]
\foreach \x in {-1,...,6}
	\foreach \y in {-1,...,6}
		{ \fill (\x,\y) circle (0.08); }
\draw[line width=1.5] (0,0)--(1,0)--(1,1)--(0,1)--(0,0);
\draw[line width=1.5] (3,0)--(4,0)--(3,1)--(2,1)--(3,0);
\draw[line width=1.5] (0,3)--(0,4)--(1,3)--(1,2)--(0,3);
\draw[dashed, line width=1] (0,5)--(0,3);
\draw[dashed, line width=1] (0,5)--(0,4);
\draw[dashed, line width=1] (0,5)--(1,3);
\draw[dashed, line width=1] (0,5)--(1,2);
\fill[yellow!50!red] (0,5) circle (0.1);
\draw[dashed, line width=1] (5,0)--(3,0);
\draw[dashed, line width=1] (5,0)--(4,0);
\draw[dashed, line width=1] (5,0)--(3,1);
\draw[dashed, line width=1] (5,0)--(2,1);
\fill[yellow!50!red] (5,0) circle (0.1);
\fill (0,0) circle (0.1);
\draw[dashed, line width=1] (0,0)--(1,1);
\node [right] at (-0.7,5.35) {$(0, \textcolor{yellow!50!red}{k}, 1)$};
\node [right] at (5,0.3) {$(\textcolor{yellow!50!red}{k}, 0, 1)$};
\node at (-0.5,3.5) {\Large $P^{\textcolor{yellow!50!red}{k}}_1$};
\node at (3.5,-0.5) {\Large $P^{\textcolor{yellow!50!red}{k}}_2$};
\node at (0.5,-0.5) {\Large $P_3$};
\node at (-0.8,-0.4) {$(0,0,1)$};
\draw [decoration={markings,mark=at position 1 with
    {\arrow[scale=2,>=latex]{>}}},postaction={decorate}] (3,2)--(2,3);
    \node [right] at (2.5,2.5) {\Large $\varphi_{1,2}$};
\end{tikzpicture}
\caption{
\label{fig:ex_family}
Top view of the infinite family from Example~\ref{ex:special}. The arrow labeled $\varphi_{1,2}$ shows the direction of the
common projection of $P_1^k$ and $P_2^k$ onto
translates of $\D_2$.
The common projections of $P_3$ and $P_1^k$ as well as 
$P_3$ and $P_2^k$ are given by the projection along the second and
the first coordinate respectively.}
\end{figure}

All computations have been carried out using Magma \cite{BCP97} and the code can 
be found at \url{https://github.com/gabrieleballetti/mixed_degree_one}.

\subsection*{Acknowledgements}
We thank Benjamin Nill for providing the initial idea
for the project and giving helpful advice and comments
throughout the progress. Furthermore we thank Ivan Soprunov for 
valuable feedback on an intermediate version of the paper.

We also thank Akiyoshi Tsuchiya and Takayuki Hibi for the opportunity to work together on the project during the Summer School and Workshop on Lattice Polytopes in Osaka 2019. This article can be seen as a vast generalization of our first investigation on the 3-dimensional case that appeared in the note ``Families of 3-dimensional polytopes of mixed degree one" in the proceedings of that conference
(\cite{proceedings19}).

The first author was partially supported by the Vetenskapsr\aa det grant NT:2014-3991. The second author is funded by the Deutsche Forschungsgemeinschaft (DFG, German Research Foundation) - 314838170, GRK 2297 MathCoRe. 

\section{Cayley sums and projections}
\label{toolbox}

While there is a well-known way of determining whether two lattice
polytopes are equivalent by comparing their normal forms
(see e.g. \cite{KS98}), the task of checking two tuples for equivalence seems 
a priori more complicated. However, we will show that in our setting
the construction of the Cayley sum of a tuple allows us
to reduce the problem of checking equivalence of tuples to 
checking equivalence of two higher-dimensional lattice polytopes.
The Cayley sum (or Cayley polytope) 
construction occurs in various contexts in the
literature, for example in the construction of mixed subdivisions
of Minkowski sums (\cite{dLRS10}), in the study
of $A$-discriminants (\cite{FI16,GKZ94}) and in structural results
on lattice
polytopes of high dimension and small degree 
(\cite{dRHN11,DdRP09}).

Let us recall the definition and some basic properties of the 
Cayley sum of a tuple of lattice polytopes.

\begin{defn}
    Let $P_1, \dotsc , P_k \subset \R^n$ be lattice polytopes.
    We define the \emph{Cayley sum} $P_1 * \dots * P_k$ as 
    \begin{align*}
	    P_1 * \dots * P_k \coloneqq 
	    \conv((P_1 \times \{\bzero\}) \cup (P_2 \times \{\be_1\}) \cup \dots \cup 
        (P_k \times \{\be_{k-1}\})) \subset \R^{n+k-1}. 
    \end{align*}
    We call a Cayley sum $P_1 * \dots * P_k$ \emph{proper} if $P_i \neq 
    \emptyset$ for all $1 \leq i \leq k$.
    In this case one has $\dim(P_1 * \dots * P_k) = 
    \dim(P_1 + \dots + P_k) + k - 1$.
\end{defn}

\begin{prop}[{\cite[Proposition~2.3]{BN08}}]
	\label{cayleyEqProjection}
	Let $P \subset \R^n$ be a lattice polytope. Then the following are equivalent.
	\begin{enumerate}
		\item There exists a lattice projection $\varphi \colon \R^n \to \R^{k-1}$
		with $\varphi(P) = \Delta_{k-1}$,  
		\item there exist lattice polytopes $P_1,\dots,P_k \subset \R^{n-k+1}$ 
		such that $P \cong P_1 * \dots * P_k$.
	\end{enumerate}
	In particular if $k=n$, $P$ is the Cayley sum of $n$ segments.
\end{prop}

\begin{rem}
	\label{rem:faces}
	Let $P = P_1 * \dots * P_k$. Then all faces of $P$ are of the 
	form $F = F_1 * \dots * F_k$ for (possibly empty) faces 
	$F_i \subseteq P_i$. One has
	$\dim(F)=\dim(F_1+\dots+F_k)+l-1$ where $l$ is the number of $F_i$ that
	satisfy $F_i \neq \emptyset$. 
	In particular, each of the $P_i$ corresponds to 
	a face of $P$ which we will denote by $\hat{P_i}$. 
	Furthermore $\hat{P}_i^c$, by which we denote the convex hull of 
	all vertices of $P$ that are not contained in $\hat{P_i}$,
	is always a proper face of $P$.
\end{rem}

	Let us now formulate the lemma which allows us to determine
	equivalence of certain $n$-dimensional tuples by comparing the 
	normal forms of their Cayley sums.	
	
\begin{lem}
	\label{lem:cayley_criterion}
    Let $P_1,\dots,P_k$ and $Q_1,\dots,Q_k$ be $k$-tuples of $n$-dimensional
    lattice polytopes in $\R^n$. Assume that there is no lattice projection 
    $\varphi \colon \R^n \to \R^{k-1}$ mapping $P_1,\dots,P_k$ onto 
    translates of $\D_{k-1}$. 
    Then the $k$-tuples $P_1,\dots,P_k$ and 
    $Q_1,\dots,Q_k$ are equivalent if and only if their Cayley sums 
    $P_1 * \dots * P_k$ and $Q_1 * \dots * Q_k$ are equivalent. 
\end{lem}
\begin{proof}
	The fact that Cayley sums of equivalent $k$-tuples of 
	lattice polytopes are equivalent
	is a straightforward consequence of the definition. 
    Suppose now that the $k$-tuple $P_1,\dots,P_k$ does not admit a
    common projection onto translates of $\D_{k-1}$ and suppose that
    $P_1 * \dots * P_k$ and $Q_1 * \dots * Q_k$ are equivalent.
    Let $P_1',\dots,P_k'$ be the images of $\hat{P}_1,\dots,
    \hat{P}_k$ under the lattice-preserving transformation yielding the
    equivalence, respectively.
    We will first show that, up to renumbering, $P_i'=\hat{Q}_i$ for 
    all $1 \leq i \leq k$.
    Suppose without loss of generality $P_1' \neq \hat{Q}_i$ for all 
    $1 \leq i \leq k$.
    Note that $P_1'$ cannot properly contain any 
    $\hat{Q}_i$
    (by Remark~\ref{rem:faces} any face of $Q_1 * \dots * Q_k$ that properly
    contains a $\hat{Q}_i$ is of dimension greater or equal to $n+1$).
    This also implies that $P_1'$ cannot be disjoint to some $\hat{Q}_i$ due
    to the following. If $P_1'$ was disjoint to, say, $\hat{Q}_1$, 
    its complement $(P_1')^c$ would contain $\hat{Q}_1$. As however
    $P_1'$ does not contain any $\hat{Q}_i$, the complement 
    $(P_1')^c$ contains at least one point of $\hat{Q}_i$ for each 
    $1 \leq i \leq k$.
    Therefore, if we pick points $p_2,\dots,p_k$ in 
    $\hat{Q}_2 \setminus P_1',\dots,\hat{Q}_k \setminus P_1'$
     respectively, we have $\dim(\hat{P}_1^c) = \dim((P_1')^c) \geq 
    \dim(\hat{Q}_1 * \set{p_2} * \dots * \set{p_k}) = n+k-1$, which is
    a contradiction to $\hat{P_1^c}$ 
    being a proper face of the Cayley sum 
    $P_1 * \dots * P_k$ (see Remark~\ref{rem:faces}).
    
    As $P_1'$ was chosen arbitrarily this argumentation implies that all
    $P_i'$ have non-empty intersection with all $\hat{Q}_j$. Therefore the 
    natural projection of $Q_1 * \dots * Q_k$ onto $\Delta_{k-1}$ remains
    surjective when restricted to the affine hull of $P_i'$ for 
    each $1 \leq i \leq k$. As the affine hulls of $P'_1,\dots,P'_k$ are
    by construction pairwise parallel $n$-dimensional 
    affine subspaces of $\R^{n+k-1}$ the natural projection yields
    a lattice projection $\varphi \colon \R^n \to \R^{k-1}$ mapping 
    $P_i$ onto a translate of $\D_{k-1}$ for each $1 \leq i \leq k$.    
	This contradicts our assumption. 
    
    This shows that, if $P_1,\dots,P_k$ do not have a common projection
	onto translates of $\D_{k-1}$, any affine lattice-preserving
	transformation $U \colon \R^{n+k-1} \to \R^{n+k-1}$ mapping
	$P_1 * \dots * P_k$ to $Q_1 * \dots * Q_k$ yields (up to
	renumbering) a bijection mapping the face $\hat{P}_i$ to the
	face $\hat{Q}_i$. We may without loss of generality assume that
	$U$ preserves the origin. Restricting $U$ to a map from the affine hull
	of $P_1$ to the affine hull of $Q_1$ we obtain a linear lattice-preserving
	transformation $L \colon \R^n \to \R^n$. 
	It is straightforward to verify that restricting $U$ to the affine
	hull of any other $P_i$ results in a map
	$x \mapsto L(x)+t_i$ from $\R^n \to \R^n$ for some $t_i \in \Z^n$.	
\end{proof}

\section{Proof of the main theorem}

From Proposition~\ref{cayleyEqProjection} we can easily deduce the following lemma.
We call a facet $F$ of a lattice polytope $P \subset \R^n$ \emph{unimodular}
if $F$ is a unimodular simplex inside the affine lattice defined as the 
intersection of the affine hull of $F$ with $\Z^n$.

\begin{lem}
	\label{lem:unimodularfacets}
    Let $P \subset \R^n$ be an $n$-dimensional lattice polytope and 
    $\varphi \colon \R^n \to \R^{n-1}$ a lattice projection that projects $P$ onto $\D_{n-1}$. Then $\ker\varphi=\R e$ where $e$ is a vector parallel to an edge 
    between vertices $v_1,v_2$, where $v_1 \in F_1$ 
    and $v_2 \in F_2$ for two different unimodular facets $F_1 \neq F_2$ of $P$.
\end{lem}

We now study $n$-dimensional polytopes projecting onto
$\D_{n-1}$ along multiple directions. Recall that we denote by
$\mathcal{P}^{n-2}(\square_2)$ the $(n-2)$-fold
lattice pyramid formed over the square $\square_2 = \conv(\bzero,\be_1,\be_2,\be_1+\be_2)\subset \R^2$.

\begin{lem}
\label{projimplunimod}
Let $P \subset \R^n$ be an $n$-dimensional lattice polytope such that there
are different lattice projections 
$\varphi_1,\varphi_2 \colon \R^n \to \R^{n-1}$ that map $P$ onto 
$\D_{n-1}$.
Then $P$ is equivalent either to the unimodular simplex $\D_n$ or to $\mathcal{P}^{n-2}(\square_2)$. 
If there exists another projection $\varphi_3 \colon \R^n \to \R^{n-1}$ mapping $P$ onto $\D_{n-1}$, then 
$P$ is necessarily equivalent to $\Delta_n$.
\end{lem}
\begin{proof}
As $P$ has one projection onto $\D_{n-1}$, by Proposition~\ref{cayleyEqProjection} we may assume
that $P$ is of the form $P = I_1 * \dots * I_n$ for  
$n$ segments $I_i = [0,a_i]$ with $a_i \in \Z_{\geq 0}$. 
Two facets of $P$ are given by $\Delta_{n-1} \times \set{0}$ and 
$\set{a_1} * \dots * \set{a_{n}}$. All other facets of $P$ are of a form
we denote by $F_k$ for $1 \leq k \leq n$, which is the 
Cayley sum of all 
$I_i$ excluding $I_k$. As there exists another lattice
 projection $\varphi_2$ mapping $P$ onto $\D_{n-1}$, by Lemma~\ref{lem:unimodularfacets}
the facet $F_k$ has to be unimodular for some $1 \leq k \leq n$.
Assume without loss of generality that $F_1$ is unimodular and
$a_2=1$ and $a_3=\dots=a_{n}=0$. 
Furthermore, $a_1$ cannot be greater than one as otherwise $P$ would have an edge of
lattice length at least $2$. Therefore any projection which is not along this
edge direction could not be projecting $P$ onto $\D_{n-1}$.
If $a_1=0$, then $P$ is equivalent to $\D_n$, otherwise $a_1=1$ and $P$ is equivalent to $\mathcal{P}^{n-2}(\square_2)$.
One easily verifies that $\mathcal{P}^{n-2}(\square_2)$ does not have more than two different projections onto $\D_{n-1}$.
\end{proof}

\begin{lem} \label{intersectCylinders}
Let $S_1, S_2 \subset \R^n$ be two unimodular $n$-dimensional simplices, $u_1,u_2 \in \Z^n$ be linearly independent edge directions for $S_1$ and $S_2$ respectively, and $C_1, C_2 \subset \R^n$ be the infinite prisms $S_1 + \R u_1$ and $S_2 + \R u_2$ respectively.
Given $z \in \Z^n$, denote by $P_z$ the intersection $\conv(C_1 \cap (C_2+z) \cap \Z^n)$. 
Let $v,w \in \Z^n$ such that $P_{v}$ and $P_{w}$ are both $n$-dimensional.
Then $P_{v}$ and $P_{w}$ are the same lattice polytope up to translation.
\end{lem}

\begin{proof}
By Lemma~\ref{projimplunimod} we know that, if there exists
$v \in \Z^n$ such that $P_v$ is $n$-dimensional, then $P_v$ is 
 equivalent either to $\Delta_n$ or to $\mathcal{P}^{n-2}(\square_2)$
having two edges parallel to the directions $u_1$ and $u_2$.
In either of the two cases, we can assume $P_v$ to be exactly
$\Delta_n$ or $\mathcal{P}^{n-2}(\square_2)$.

If $P_v = \mathcal{P}^{n-2}(\square_2)$ then, up to reordering and
changes of signs, $u_1=\be_1$ and $u_2=\be_2$. In particular, $C_1 = \conv(\bzero,\be_2,\ldots,\be_n) + \R \be_1$
and $C_2 + v = \conv(\bzero,\be_1,\be_3,\ldots,\be_n) + \R \be_2$.
One easily verifies that $C_1 \cap (C_2+w)$ is full-dimensional if and only if $w-v \in \Z \be_1 + \Z \be_2$.
In all these cases $C_1 \cap (C_2+w)$ is a translation of $P_v$.

On the other hand, if $P_v=\D_n$, then there is another case distinction.
If $u_1$ and $u_2$ are parallel to adjacent edges of $D_n$, then we can assume $u_1=\be_1$ and $u_2=\be_2$.
But in this case $C_1$ and $C_2+v$ must intersect in $\mathcal{P}^{n-2}(\square_2)$ instead of in $\D_n$, hence we have a contradiction. Therefore $u_1$ and $u_2$ are parallel to non-adjacent edges of $\D_n$ and we
can assume $u_1 = \be_1$ and $u_2 = \be_2-\be_3$. In particular, $C_1 = \conv(\bzero,\be_2,\ldots,\be_n) + \R \be_1$
and $C_2 + v = \conv(\bzero,\be_1,\be_3,\ldots,\be_n) + \R (\be_2 - \be_3)$.
Again, one easily verifies that $C_1 \cap (C_2+w)$ is full-dimensional if and only if $w-v \in \Z \be_1 + \Z (\be_2-\be_3)$.
In all these cases $C_1 \cap (C_2+w)$ is a translation of $\D_n$.
\end{proof}

\begin{lem}
\label{lem:two_simplices}
Let $P \subset \R^n$ be a unimodular $n$-simplex and $\varphi_1,\varphi_2 \colon \R^n \to \R^{n-1}$
be two different lattice projections such that, for each
$1 \leq i \leq 2$, the images $\varphi_i(P)$ and $\varphi_i(\Delta_n)$
are translates of $\D_{n-1}$.
Then, up to translation and coordinate permutation, $P$ is contained in $\mathcal{P}^{n-2}(\square_2)$.
If there exists another projection $\varphi_3 \colon \R^n \to \R^{n-1}$ mapping $P$ and $\Delta_n$ onto translates of $\D_{n-1}$, then $P$ is necessarily a translate of $\D_n$.
\end{lem}
\begin{proof}
By Lemma~\ref{lem:unimodularfacets}, $\varphi_1$ and $\varphi_2$ are projections along the directions $u_1$ and $u_2$ of two edges of $\D_n$.
If $u_1$ and $u_2$ are the directions of two adjacent edges of $\D_n$, we can suppose that $u_1=\be_1$ and $u_2=\be_2$.
Then $P$ is contained in the intersection $(C_1+z_1) \cap (C_2+z_2)$ where
$C_1\coloneqq\D_n + \R \be_1$ and $C_2\coloneqq\D_n + \R \be_2$, for some $z_1,z_2 \in \Z^n$.
By Lemma~\ref{intersectCylinders}, $P$ is, up to translation, contained in $C_1 \cap C_2 = \mathcal{P}^{n-2}(\square_2)$.
If $u_1$ and $u_2$ are the directions of two non-adjacent edges of $\D_n$
then we can suppose that $u_1=\be_1$ and $u_2=\be_2-\be_3$.
Then $P$ is contained in the intersection $(C_1+z_1) \cap (C_2+z_2)$ where
$C_1\coloneqq\D_n + \R \be_1$ and $C_2\coloneqq\D_n + \R(\be_2-\be_3)$, for some $z_1,z_2 \in \Z^n$.
By Lemma~\ref{intersectCylinders}, $P$ is, up to translation, contained in $C_1 \cap C_2 = \D_n$,
therefore $P$ is a translate of $\D_n \subset \mathcal{P}^{n-2}(\square_2)$.
This proves the first part of the statement.

For the second part of the statement we note that $\varphi_3$ must also be a
projection along the direction $u_3$ of an edge of $\D_n$. The only case we need to check is when the edges parallel to
$u_1$, $u_2$ and $u_3$ form a triangle in $\D_n$.
Indeed, if this is not the case either two of these edges are non-adjacent
and $P$ must be a translate of $\D_n$ as above, or $u_1$,$u_2$ and $u_3$ share
a vertex. In the latter case we may assume $u_i = \be_i$ for $1 \leq i \leq 3$. 
As deduced above from Lemma~\ref{intersectCylinders}, this in
particular yields that $P$ is contained in the intersection of a translation 
of the square pyramid $\conv(\bzero,\be_1,\be_2,\be_1+\be_2,\be_3,\dots,\be_n)$ with the flipped square pyramid $\conv(\bzero,\be_1,\be_3,\be_1+\be_3,
\be_2,\be_4,\dots,\be_n)$. This 
implies that $P$ is a translate of $\D_n$.
Let us therefore assume that $u_1 = \be_1$, $u_2 = \be_2$ and $u_3 = \be_1 - \be_2$.
In this case $P$ is a translate of one of the four $n$-dimensional subpolytopes of $\mathcal{P}^{n-2}(\square_2)$.
It is easy to verify that $\D_n$ is the only one of them that is projected
by $\varphi_3$ onto a translate of $\varphi_3(\D_n)$.
\end{proof}

\begin{defn}
Let $P_1,\dots,P_{n-1} \subset \R^n$ be $n$-dimensional polytopes with the Minkowski
sum $P_1 + \dots + P_{n-1}$ being hollow. We call the $(n-1)$-tuple
$P_1,\dots,P_{n-1}$ \emph{exceptional}, if there exists no projection
$\varphi \colon \R^n \to \R^{n-1}$ such that 
$\varphi(P_1 + \dots + P_{n-1}) \subset \R^{n-1}$
is a hollow polytope.
\end{defn}

\begin{rem}
\label{rem:exc-nonexc}
By \cite[Theorem~1.2]{NZ11} there exist only finitely many $n$-dimensional lattice polytopes not admitting a lattice projection onto a hollow $(n-1)$-dimensional
lattice polytope, up to equivalence. So in particular, up to equivalence, there exist only finitely many exceptional $(n-1)$-tuples of $n$-dimensional lattice polytopes.

Furthermore, by Proposition~\ref{prop:ccd}, for any non-exceptional $(n-1)$-tuple $P_1,\dots,P_{n-1}$ of $n$-dimensional lattice polytopes there exists a lattice projection 
$\varphi \colon \R^n \to \R^{n-1}$ mapping all $P_i$ onto translates
of $\D_{n-1}$.
\end{rem}

\begin{proof}[Proof of Theorem~\ref{thm:main}]
Let $n \geq 4$.
Given $1 \leq k \leq n$, denote by $I_k$ the set $\set{1,...,n} \setminus \{k\}$, and by
$[P]_k$ the $(n-1)$-tuple given by all
$P_i$ for $i \in I_k$. Denote furthermore by $P_{I_k}$ the 
Minkowski sum $\sum_{i \in I_k} P_i$ of the polytopes
in $[P]_k$.
Since $\md(P_1,\dots,P_n)=1$, the Minkowski sum $P_{I_k}$ is hollow
for any $1 \leq k \leq n$.
Recall that, if $[P]_k$ is not exceptional, then
by Remark~\ref{rem:exc-nonexc} there exists a projection 
$\varphi:\R^n \to \R^{n-1}$ mapping all polytopes in $[P]_k$ onto
translates of $\D_{n-1}$.
We treat cases separately, depending on the number of exceptional $(n-1)$-subtuples of the tuple $P_1,\ldots,P_n$.
\begin{enumerate}
\item[(0)] If $P_1,\ldots,P_n$ has no exceptional $(n-1)$-subtuples then either there
	exists a projection $\varphi:\R^n \to \R^{n-1}$ mapping $P_1,\ldots,P_n$ onto
	translates of $\D_{n-1}$ (and in this case there is nothing to prove),
	or each of the $P_i$ admits $n-1$ pairwise different projections onto 
	$\D_{n-1}$.
	Indeed if two of these projections were the same, then we would be in the previous case.
	Suppose there exist $n-1$ pairwise different projections. As $n \geq 4$, 
	Lemma~\ref{projimplunimod} yields
	that each of the $P_i$ is a unimodular $n$-dimensional simplex. Without loss of generality we assume $P_1 = \D_n$. 
	Given $2 \leq i \leq n$, there exist $n-2$ pairwise different projections mapping $P_1$ and $P_i$
	onto translates of $\D_{n-1}$. If $n \geq 5$, by Lemma~\ref{lem:two_simplices}, we 
	can immediately deduce that, up to translations, $P_1=P_2= \ldots=P_n =\D_n$. If $n=4$, Lemma~\ref{lem:two_simplices}
	only ensures that $P_2,\ldots,P_n$ are, up to translation and coordinate permutation, contained in $\mathcal{P}^{n-2}(\square_2)$. This yields 
	finitely many cases and checking them computationally we find
	among them only $4$-tuples
	admitting a common projection onto $\D_3$.
	
\item[(1)] $P_1,\ldots,P_n$ has exactly one exceptional $(n-1)$-subtuple, which we can assume to be $[P]_n$. As $[P]_n$ is an exceptional $(n-1)$-tuple, the Minkowski sum $P_{I_n}$ belongs to a finite list of hollow $n$-dimensional polytopes. This means that there are, up to equivalence,
	finitely many exceptional tuples to choose $[P]_n$ from. 
	We now show, that given
	$[P]_n$ there are finitely many possible choices for $P_n$ that lead to
	the $n$-tuple $P_1,\dots,P_n$ having exactly $[P]_n$ as an
	exceptional $(n-1)$-subtuple, which shows the finiteness of this case.
	
	Let therefore $\varphi_1:\R^n \to \R^{n-1}$ be a lattice projection mapping the lattice polytopes in $[P]_2$ to
	translates of $\D_{n-1}$. Similarly, let $\varphi_2:\R^n \to \R^{n-1}$ be a
	lattice projection mapping the lattice polytopes in $[P]_1$ to translates 
	of $\D_{n-1}$.
	The existence of such projections follows from the fact that $[P]_2$ and $[P]_1$ are non-exceptional.
	We remark that there exist finitely many such projections.
	Let $C_i$ be the infinite prism $P_i + \ker \varphi_i$, for $1 \leq i \leq 2$. Then we know that any possible choice of $P_n$ is contained in $(C_1 + z_1) \cap (C_2 + z_2)$ for some $z_1, z_2 \in \Z^n$.
	By Lemma~\ref{intersectCylinders}, for any choices of lattice points
	$z_1,z_2,z_1',z_2' \in \Z^n$ such that $\dim((C_1 + z_1) \cap (C_2 + z_2) \cap \Z^n) = \dim((C_1 + z_1') \cap (C_2 + z_2') \cap \Z^n) =  n$ we find that $(C_1 + z_1') \cap (C_2 + z_2')$ is a translate
	of $(C_1 + z_1) \cap (C_2 + z_2)$.
	Therefore, up to translations, all possible choices for $P_n$ are contained in $(C_1 + z_1) \cap (C_2 + z_2)$ for fixed $z_1, z_2 \in \Z^n$. Note that the intersection $(C_1 + z_1) \cap (C_2 + z_2)$ is either equivalent to $\D_n$
	or $\mathcal{P}^{n-2}(\square_2)$ by Lemma~\ref{projimplunimod}, where the choice of the equivalence class depends entirely on $[P]_n$.
	This implies that $P_n$ must be one element of a finite list of lattice polytopes fully determined by $[P]_n$.
\item[(2+)] If $P_1,\ldots,P_n$ has two or more exceptional $(n-1)$-subtuples, 
	then we can suppose that $[P]_n$
	and $[P]_{n-1}$ are exceptional. In particular, there exists an upper bound depending only on $n$ for the volume of 
	the Minkowski sums $P_{I_n}$ and $P_{I_{n-1}}$
	and therefore (since $n>2$) for the volume of $P_1 + P_i$ for any $2 \leq i \leq n$. Recall that, by \cite[Theorem~2]{LZ91}, there are, up to equivalence, only finitely many lattice polytopes of any fixed volume $K \in \Z_{\geq 0}$.
	Therefore, as in particular the volume of $P_1$ is bounded, there exist only finitely many choices for $P_1$
	up to equivalence. Furthermore, fixing $P_1$ determines, up to translation,
	finitely many possibilities for each $P_i$ with $2 \leq i \leq n$
	due to the volume bound on $P_1+P_i$. This yields that there are
	only finitely many $n$-tuples $P_1,\dots,P_n$ in this case,
	up to equivalence.
\end{enumerate}
\end{proof}

Note, that the assumption $n > 3$ is only used in case~$(0)$ of the previous proof. 

The unmixed result of Theorem~\ref{thm:BN}
also gives an explicit description of lattice polytopes of degree one that 
are not Lawrence prisms, in fact, up to equivalence and the lattice pyramid construction, there exists only one such exception over 
all dimensions. Such an explicit description of the
list of exceptions from the statement of Theorem~\ref{thm:main} 
is not known in dimension $n \geq 4$.

\begin{quest}
\label{question}
For dimension $n \geq 4$, what are the $n$-tuples of
$n$-dimensional lattice polytopes 
$P_1,\dots,P_n \subset \R^n$ with $\md(P_1,\dots,P_n)=1$ that are not
of the trivial type described in Example~\ref{ex:lawrence}? Is there a finite description over all dimensions as there is in the unmixed case?
\end{quest}

In \cite{Nil17} the mixed degree is actually treated in a more general way, also being defined for $m$-tuples of $n$-dimensional lattice 
polytopes, with $m \neq n$. In particular, an $m$-tuple
of lattice polytopes $P_1,\dots,P_m \subset \R^n$
satisfies $\md(P_1,\dots,P_m) \leq 1$ if and only if $m \geq n-1$ and
the Minkowski sum of each $(n-1)$-subtuple is hollow. For $m = n-1$ 
we obtain an analogous result to Theorem~\ref{thm:main} 
(even for $n \in \set{2,3}$) immediately from \cite[Theorem~1.2]{NZ11}.
We remark that Theorem~\ref{thm:main} also
inductively extends to the case of $m > n$ as follows.

\begin{rem}
Fix $n \geq 4$ and let $P_1,\dots,P_{n+k} \subset \R^n$ be $n$-dimensional 
lattice polytopes with $\md(P_1,\dots,P_{n+k})=1$. Then, up to equivalence,
the $(n+k)$-tuple $P_1,\dots,P_{n+k}$ either belongs to a finite list of 
exceptions or there is a lattice projection 
$\varphi \colon \R^n \to \R^{n-1}$ such that $\varphi(P_i) = \D_{n-1}$
for all $1 \leq i \leq n+k$.

One can see this with an induction argument on $k$, where the base case
is given by Theorem~\ref{thm:main}.
Indeed, let $P_1,\dots,P_{n+k+1} \subset \R^n$ be an $(n+k+1)$-tuple of 
$n$-dimensional lattice polytopes with $\md(P_1,\dots,P_{n+k+1})=1$. One easily verifies that this implies that
any $(n+k)$-subtuple of $P_1,\dots,P_{n+k+1}$ has mixed
degree at most $1$.
Analogously to the proof of Theorem~\ref{thm:main} one can distinguish
three cases depending on how many $(n+k)$-subtuples of 
$P_1,\dots,P_{n+k+1}$ do not admit a common lattice projection onto 
translates of $\D_{n-1}$, and use the induction hypothesis.
\end{rem}

\section{The $3$-dimensional case}

This section is devoted to the proof of Theorem~\ref{thm:classification3d},
giving a classification of triples of $n$-dimensional lattice 
polytopes
$P_1,P_2,P_3 \subset \R^3$ of mixed degree one. Note that from the proof of 
Theorem~\ref{thm:main} it follows that the number of such triples is finite, 
if we assume at least one of the 
subpairs of $P_1,P_2,P_3$ to be exceptional. Here we first classify, up to equivalence, 
these finitely many triples. In Proposition~\ref{prop:families} we show that there are non-trivial
infinite $1$-parameter families of triples.

As an intermediate step towards the classification of triples of lattice polytopes of
mixed degree one with at least one exceptional subpair we
calculate all (equivalence classes of) exceptional pairs of 
$3$-dimensional lattice polytopes. In order to do that we consider the list of
maximal hollow $3$-dimensional lattice polytopes classified by Averkov--Wagner--Weismantel
\cite{AWW11} (see also \cite{AKW17}), and compute all subpolytopes of the
maximal hollow lattice polytopes that have lattice width greater than one. 

\begin{prop}[{\cite[Corollary~1]{AWW11}}]
	\label{prop:hollows}
    Let $P \subset \R^3$ be a hollow $3$-dimensional lattice polytope
    of width at least two.
    Then, up to equivalence, $P$ is contained either in the 
    unbounded polyhedron $2\Delta_2 \times \R$ or in one of 12 maximal
    hollow lattice polytopes.
\end{prop}

As we are interested in obtaining a list of exceptional pairs $P,Q \subset \R^3$
we use an implementation in Magma in order to compute the
decompositions of all subpolytopes of the $12$ maximal hollow polytopes into
Minkowski sums of two $3$-dimensional lattice polytopes.
Afterwards we determine those pairs that actually do not admit 
a common
projection onto translates of $\D_2$ and then determine
equivalent pairs using Lemma~\ref{lem:cayley_criterion}.

\begin{cor}
	\label{cor:couples}
    There are, up to equivalence, 32 pairs of $3$-dimensional
    lattice polytopes whose Minkowski sum is hollow and that do
    not admit a common projection onto translates of $\D_2$.
\end{cor} 

We use this classification in order to compute all triples of 
lattice polytopes $P_1,P_2,P_3 \subset \R^3$ of mixed degree one with at least
two exceptional subpairs as follows.

Assume that $P_1,P_2$ and $P_1,P_3$ are exceptional pairs.
Then there exist two pairs $A,B$ and $C,D$ out of the 32 of Corollary~\ref{cor:couples}
such that $A,B$ is equivalent to $P_1,P_2$ and $C,D$ is equivalent to $P_1,P_3$.
We can suppose that $P_1$ is equal to $A$ and equivalent to $C$.
Thus there exists an affine lattice-preserving transformation $\varphi$ mapping $C$ to $A=P_1$
such that the triple $P_1,P_2,P_3$ is equivalent to the triple $A,B,\varphi(D)$.

This justifies the following algorithm to construct all the triples $P_1,P_2,P_3$
containing at least two exceptional subpairs: we iterate over all the pairs of
ordered pairs $A,B$ and $C,D$ of Corollary~\ref{cor:couples}, and, whenever there exists
an affine lattice-preserving transformation $\varphi$ mapping $C$ to $A$, check if the triple $A,B,\varphi(\psi(D))$
has mixed degree one, where $\psi$ ranges among all the possible affine automorphisms of $C$ (and therefore $\varphi \circ \psi$
ranges among all affine lattice-preserving transformations sending
$C$ to $A$).
Equivalent triples can be removed using the criterion following from Lemma~\ref{lem:cayley_criterion}.
An implementation in Magma yields the following result proving parts (i)-(ii) of Theorem~\ref{thm:classification3d}.

\begin{prop}
    There are, up to equivalence, 170 triples of $3$-dimensional
     lattice polytopes
    of mixed degree one having two or three exceptional subpairs. In the first case there are 29
    triples, in the latter there are 141.
\end{prop}

We now discuss the case of triples of lattice polytopes of mixed degree
one having exactly one exceptional subpair.
Specifically, $P_1,P_2,P_3 \subset \R^3$ is a triple of $3$-dimensional
lattice polytopes with $\md(P_1,P_2,P_3)=1$ and (without loss of generality)
there are two different lattice projections $\varphi_{2} \colon \R^3 \to \R^2$ and
$\varphi_{3} \colon \R^3 \to \R^2$ where $\varphi_{k}$ maps $P_i$ and $P_j$
to translates of $\D_2$ whenever $i,j \neq k$.
In particular $P_1$ is a lattice polytope with two different 
lattice projections onto $\D_2$ (and therefore by Lemma~\ref{projimplunimod} is
equivalent either to $\D_3$ or to $\sqp$) and $P_2,P_3$ is an exceptional pair. 
Note that $P_1$ must be contained in both the infinite prisms $C_2 \coloneqq P_3 + \ker \varphi_2 + u$ and
$C_3 \coloneqq P_2 + \ker \varphi_3 + v$, for some translation vectors $u,v \in \Z^3$.

In order to classify all such triples we use the fact that we
may choose $P_2,P_3$ from the list of 32 exceptional pairs of
 Corollary~\ref{cor:couples}.
Given an exceptional pair $P_2,P_3$, we iterate over all the possible pairs of lattice projections $\varphi_3,\varphi_2$,
such that $\varphi_3(P_2)$ and $\varphi_2(P_3)$ are unimodular triangles.
Each such choice determines two infinite prisms $C_3 \coloneqq P_2 + \ker \varphi_3$ and 
$C_2 \coloneqq P_3 + \ker \varphi_2$. We know that any lattice polytope $P_1 \subset \R^3$,
such that $\varphi_3(P_1)$ and $\varphi_2(P_1)$ are translates
of $\varphi_3(P_2)$ and $\varphi_2(P_3)$ respectively, is contained in both the
infinite prisms $C_2 \coloneqq P_3 + \ker \varphi_2 + u$ and
$C_3 \coloneqq P_2 + \ker \varphi_3 + v$, for some translation vectors $u,v \in \Z^3$. Up to translation of $P_1$ we may assume $u=\bzero$.
By Lemma~\ref{intersectCylinders} it suffices to find one choice
of $v \in \Z^3$ such that $C_2$ and $C_3$ intersect in a full-dimensional lattice polytope, in order to determine the 
inclusion-maximal choice for $P_1$ up to translation.
Furthermore, there are only finitely many choices for 
$v \in \Z^3$ to check for the existence of a full-dimensional
 intersection of 
$C_2$ and $C_3$ as we may suppose $P_2$ and $P_3$ to have a common vertex. This is due to the fact that, if $C_2$ and $C_3$ intersect
in a full-dimensional lattice polytope, then one may translate
$P_2$ along $\ker \varphi_3$ and $P_3$ along $\ker \varphi_2$ without
changing the infinite prisms. It therefore suffices to restrict to translation vectors
$v$ that map a vertex of $P_2$ to a vertex of $P_3$.
Thus we can determine, up to equivalence, all inclusion-maximal
$P_1$ as above, form triples for all subpolytopes of $P_1$ and remove equivalent triples using Lemma~\ref{lem:cayley_criterion}.
An implementation in Magma yields the following result proving part (iii) of Theorem~\ref{thm:classification3d}.

\begin{prop}
    There are, up to equivalence, 82 triples of $3$-dimensional
     lattice polytopes
    of mixed degree one having exactly one exceptional subpair.
\end{prop}

In the remaining part of this section we are going to deal with non-trivial triples not having any exceptional subpair in order to prove 
part (iv) of Theorem~\ref{thm:classification3d}.

\begin{lem}
\label{lem:proj_span_lat}
Let $P_1, P_2, P_3 \subset \R^n$ be lattice polytopes, and $\varphi_1,\varphi_2,\varphi_3 \colon \R^3 \to \R^2$
be lattice projections such that, for all $1 \leq i,j,k \leq 3$, the images $\varphi_k(P_i)$ and $\varphi_k(P_j)$ are translates of $\D_2$ if and only if $i,j \neq k$. Let $v_i \in \Z^3$ be the projection direction
of $\varphi_i$ for $1 \leq i \leq 3$. Then $v_i$ and $v_j$ are part of a lattice basis of $\Z^3$, for all $1 \leq i,j \leq 3$. Moreover, if $v_1,v_2,v_3$ linearly span $\R^3$, then they form a lattice basis of $\Z^3$.
\end{lem}
\begin{proof}
For $1 \leq k \leq 3$ let $C_k$ be the infinite prism $\varphi_k(P_i) + \R v_k$, for some $i \neq k$. Note that, up to translation,
this does not depend on the choice of $i$ as both $P_i$ and $P_j$ are contained in different
translates of $C_k$, whenever $i,j \neq k$. We now fix any of the infinite prisms, say $C_1$.
For simplicity we suppose $C_1 = ( \{0\} \times \D_2) + \R \be_1$ and $P_2, P_3 \subset C_1$. In this way we avoid dealing with translations.
Note that $v_2$ is parallel to an edge of $P_3$, and $v_3$ is parallel to an edge of $P_2$.
Since both edges are contained in $C_1$, they project along $\be_1$ either to the same side
of the triangle $\varphi_1(P_2)=\varphi_1(P_3)=\D_2$, or to two adjacent sides.
In the second case $\be_1$, $v_2$ and $v_3$ linearly span
$\R^3$ and it is easy to verify that they form a lattice basis of $\Z^3$.
In the first case $\be_1$, $v_2$ and $v_3$ span a plane, and from Lemma~\ref{projimplunimod} it follows
that any two of them are part of a lattice basis of $\Z^3$.
\end{proof}

\begin{prop}
\label{prop:directions_span}
There are, up to equivalence, 27 triples of
$3$-dimensional lattice polytopes $P_1,P_2,P_3 \subset \R^3$ satisfying the 
hypotheses of Lemma~\ref{lem:proj_span_lat} for projection directions $v_1,v_2,v_3 \in \Z^3$ that linearly span $\R^3$.
All of them are, up to equivalence, contained in one of the
following three inclusion-maximal triples of mixed degree one:
\begin{itemize}
\item the maximal triple given by the following three reflections of $\sqp$
		\begin{itemize}[label={}]
			\item $\conv(\bzero,\be_2, \be_3, \be_2 + \be_3, \be_1)$,
			\item $\conv(\bzero,\be_1, \be_3, \be_1 + \be_3, \be_2)$,
			\item $\conv(\bzero,\be_1, \be_2, \be_1 + \be_2, \be_3)=\sqp$,
		\end{itemize}
\item the maximal triple
		\begin{itemize}[label={}]
			\item $\conv(\bzero, \be_1, \be_3, \be_1 + \be_2)$,
			\item $\conv(\bzero,\be_1,\be_3,\be_1 + \be_2, \be_1 + \be_3)$,
			\item $\conv(\bzero,\be_1, \be_2, \be_1 + \be_2, \be_3)=\sqp$.
		\end{itemize}
\item and the maximal triple
		\begin{itemize}[label={}]
			\item $\conv(\be_1, \be_2, \be_1 + \be_2, \be_2 + \be_3)$,
			\item $\conv(\be_1, \be_2, \be_1 + \be_2, \be_1 + \be_3)$,
			\item $\conv(\bzero,\be_1, \be_2, \be_1 + \be_2, \be_3)=\sqp$.
		\end{itemize}
\end{itemize}
\end{prop}
\begin{proof}
By Lemma~\ref{lem:proj_span_lat} we may assume $v_1,v_2,v_3$ to be $\be_1,\be_2,\be_3$ respectively,
and that two primitive segments parallel to the directions $\be_2$ and $\be_3$ are contained in $C_1$. 
This restricts $C_1$ to be, up to translation, one of the four infinite prisms of the form $\conv(\bzero,\pm \be_2, \pm \be_3) + \R \be_1$.
In particular, up to translation, $C_1$ is contained in the infinite prism 
$\conv(\bzero,\be_2) + \conv(\bzero,\be_3) + \R \be_1$.
Similarly, $C_2 \subset \conv(\bzero,\be_1) + \conv(\bzero,\be_3) + \R \be_2$ and $C_3 \subset \conv(\bzero,\be_1) + \conv(\bzero,\be_2) + \R \be_3$.
In particular all the $P_i$ are, up to translations, subpolytopes of the unit cube $\square_3 = \conv(\bzero,\be_1)+\conv(\bzero,\be_2)
+\conv(\bzero,\be_3)$,
which leaves finitely many cases that we check computationally.
\end{proof}

From the proof of Proposition~\ref{prop:directions_span} it is clear that
all the maximal triples from Proposition~\ref{prop:directions_span} are
actually contained inside the triple consisting of three copies of the
unit cube $\square_3$. Note however that one has 
$\md(\square_3,\square_3,\square_3)>1$, as the Minkowski sum 
$\square_3+\square_3$ has an interior lattice point.

\begin{prop}
\label{prop:families}
There are, up to equivalence, infinitely many triples of
$3$-dimensional lattice polytopes $P_1,P_2,P_3 \subset \R^3$ satisfying the 
hypothesis of Lemma~\ref{lem:proj_span_lat} for projection directions $v_1,v_2,v_3 \in \Z^3$ that linearly span 
$\R^2 \times \set{0}$. All of them, up to equivalence, are contained in one of the following triples of mixed degree two
given by the parallelepipeds $Q_k,R_k,\square_3$ for some $k \in \Z_{\geq 0}$, where
\begin{align*}
	Q_k \coloneqq & \conv(\bzero,\be_1-\be_2) + \conv(\bzero,\be_2) + \conv(\bzero,k \be_2 + \be_3),\\
	R_k \coloneqq & \conv(\bzero,\be_2-\be_1) + \conv(\bzero,\be_1) + \conv(\bzero,k \be_1 + \be_3),\\
	\square_3 \coloneqq & \conv(\bzero,\be_1) + \conv(\bzero,\be_2) + \conv(\bzero, \be_3).
\end{align*}
They can be covered by a finite
number of $1$-parameter families.
In particular, if we denote by $\psi_1^k$ the shearing $(x,y,z) \mapsto (x,y+kz,z)$ and
by $\psi_2^k$ the shearing $(x,y,z) \mapsto (x+kz,y,z)$, one
may choose $1$-parameter families of the form 
$\set{\psi_1^k(P^0_1),\psi_2^k(P^0_2),P_3}_{k \in \Z_{\geq 0}}$
for all subpolytopes
$P_1^0 \subset Q_0,P_2^0 \subset R_0$ and $P_3 \subset \square_3$
satisfying $\md(P_1^0,P_2^0,P_3) = 1$. 
\end{prop}
\begin{proof}
By Lemma~\ref{lem:proj_span_lat} we may assume $v_1,v_2,v_3$ to be $\be_1,\be_2,\be_1-\be_2$. Here, the assumption $v_3 = \be_1 - \be_2$
follows from the fact that both the pairs $\be_1,v_3$ and $\be_2,v_3$ need to be part of a lattice basis of $\Z^3$,
and the projection directions $v_i$ may be chosen with arbitrary sign.
By Lemma~\ref{projimplunimod} the polytope $P_3$, which projects
onto $\D_2$ along the directions $\be_1$ and $\be_2$, can be fixed to be in the unit cube $\square_3$.
Consequently we can assume $C_1,C_2$ to be in the infinite prisms $\square_3 + \R \be_1$ and $\square_3 + \R \be_2$, respectively.
Finally, we assume $P_1$ and $P_2$ to be in the infinite prisms $C_2$ and $C_1$, respectively.
Now consider the linear functional $f$ defined by $(x,y,z) \mapsto x+y$. Consider a lattice point
$v_0 \in P_1 \cap (\R^2 \times \{0\})$ minimizing $f$. Since $P_1$ projects onto
$\D_2$ along the
direction $\be_1 - \be_2$, one verifies that for any other point $u_0 \in P_1 \cap (\R^2 \times \{0\})$ one has
$f(v_0) \leq f(u_0) \leq f(v_0) + 1$.
Analogously, if $v_1$ is a lattice point in $P_1 \cap (\R^2 \times \{1\})$ minimizing $f$, then
$f(v_1) \leq f(u_1) \leq f(v_1) + 1$ for all $u_1 \in P_1 \cap (\R^2 \times \{1\})$. 
Since we are free to translate $P_1$ along $\be_2$, we can suppose $f(v_0)=0$ and we denote $k=f(v_1)$.
As a consequence, $P_1 \cap (\R^2 \times \{0\})$ is contained in the parallelogram
$q_0 \coloneqq \conv(\bzero, \be_1 - \be_2) + \conv(\bzero, \be_2)$. Analogously $P_1 \cap (\R^2 \times \{1\})$
is contained in the parallelogram
$q_1 \coloneqq q_0 + k \be_2 + \be_3$.
In particular $P_1$ is contained in the parallelepiped $\conv (q_0 \cup q_1) = Q_k$.
Therefore $C_3$ is contained in the infinite prism $Q_k + \R(\be_1 - \be_2)$.
This completely determines the parallelepiped $R_k = C_1 \cap C_3$, satisfying $R_k \supset P_2$.
It is easy to verify that the triple $Q_{k},R_{k},\square_3$ is equivalent to the triple $Q_{-k},R_{-k},\square_3$, so one
can always assume $k \in \Z_{\geq 0}$.  

In order to see that the set of triples that are subtriples
of $Q_k,R_k,\square_3$ for some $k \in \Z_{\geq 0}$ can be covered
by $1$-parameter families as claimed it suffices to notice that any 
subtriple of $Q_k,R_k,\square_3$ can be written as 
$\psi_1^k(P_1^0),\psi_2^k(P_2^0),\square_3$ for subpolytopes
$P_1^0 \subset Q_0,P_2^0 \subset R_0$ and $P_3 \subset \square_3$.
The fact that any family $\set{\psi_1^k(P^0_1),\psi_2^k(P^0_2),P_3}_{k \in \Z_{\geq 0}}$ for subpolytopes
$P_1^0 \subset Q_0,P_2^0 \subset R_0$ and $P_3 \subset \square_3$ actually
contains infinitely many non-equivalent triples can be verified by picking edges $E_i \subset P_i$ for $1 \leq i \leq 3$ between vertices on
height 0 and 1, and noticing that the volume of the parallelepiped $E_1 + E_2 + E_3 $ grows quadratically in $k$.
An example of one of these inifinite $1$-parameter families is given in Example~\ref{ex:special}.
\end{proof}

A computer assisted search for mixed degree one triples in $Q_{k},R_{k},\square_3$ for small values of $k$ shows that there are
51 non-equivalent triples when $k=0$, and $36$ for larger values of $k$, where, for each $k$, the overlaps that occur for preceding values of $k$ are excluded.

Let us finally mention that the 252 triples classified in
Theorem~\ref{thm:classification3d} \emph{(i)} --\emph{(iii)} can
all be found as subtriples of six special triples which are 
maximal with respect to inclusion. We have verified this
computationally by enumerating subtriples of the six special ones. 

\begin{cor}
    \label{thm:maximal}
    All triples $P_1, P_2, P_3 \subset \R^3$ of $3$-dimensional
    lattice polytopes of mixed degree one of types (i) --(iii) from
    Theorem~\ref{thm:classification3d} are, up to equivalence, contained in one of the following 
    $6$ maximal triples:
    \begin{enumerate}[label=(\alph*)]
        \item the maximal triple $\Pm(2\Delta_2),\Pm(2\Delta_2),\Pm(2\Delta_2)$,
        \item the maximal triple $2 \Delta_3,\Delta_3,\Delta_3$,
        \item the maximal triple $\{\conv(\bzero,2\be_i,\be_j,\be_k) \colon i,j,k \in [3]
        \text{ pairwise different}\}$,
        \item the maximal triple
        \label{itm:max_triple_picture}
		\begin{align*}		        
        &\conv(\be_1,\be_2,-\be_2)*\conv(\bzero,\be_1), \\ 
        &\conv(\bzero,\be_1,-\be_2) * \{-\be_2\}, \\
        &\conv(\bzero,\be_1,\be_2)*\{\be_2\},
        \end{align*}
		\item the maximal triple
		\begin{align*}
		&\conv(\bzero,2\be_2) * \conv(\bzero,\be_1), \\
		&\conv(\bzero,-\be_1,-\be_1-\be_2) * \{-\be_1-2\be_2\}, \\
		 &\conv(\bzero,\be_2,-\be_1) * \{\be_1\},
		\end{align*}
		\item the maximal triple
		\begin{align*}
		&\conv(\bzero,2\be_2) * \conv(\bzero,\be_1), \\
		&\conv(\bzero,-\be_1,-\be_1-\be_2) * \{-\be_1-2\be_2\}, \\
		 &\conv(\bzero,-\be_2,-\be_1)*\{\be_1-2\be_2\}.
		\end{align*}
	\end{enumerate}    
\end{cor}

Note that the maximal triples (a) and (b) of
Corollary~\ref{thm:maximal} admit direct generalizations to an arbitrary dimension $n$ that are of mixed degree one. 
Furthermore it follows from the proof of
Theorem~\ref{thm:main} that there are no $n$-tuples of
a type analogous to type (iv) of Theorem~\ref{thm:classification3d}
for $n \geq 4$. A bold guess for an answer to Question~\ref{question} would be that for arbitrary $n \geq 4$ (or 
$n$ large enough) all exceptions of $n$-tuples
of mixed degree one are contained in one of these generalizations, that is 
either in $\Pm^{n-2}(2\Delta_2),\dots,\Pm^{n-2}(2\Delta_2)$ or $2 \D_n, \D_n, \dots, \D_n$ (as it can be easily verified that the 
straightforward generalization of the maximal family (c) to dimension
$n \geq 4$ does not yield $n$-tuples of mixed degree one).


\bibliographystyle{alpha}
\bibliography{mixed_degree}
\end{document}